\documentclass[11pt,twoside]{amsart}
 \usepackage[T1]{fontenc}
\usepackage[utf8]{inputenc}
\usepackage[usenames]{color}
\input{xy}
\xyoption{all}
\xyoption{poly}
\usepackage[all]{xy}

 \usepackage{amsmath,amsthm,amsfonts,amssymb}
      \theoremstyle{plain}
      \newtheorem{theorem}{Theorem}[section]
      \newtheorem*{theorem*}{Theorem}
      \newtheorem{lemma}[theorem]{Lemma}
      
      \newtheorem{proposition}[theorem]{Proposition}
      
      \newtheorem*{aproposition}{Proposition \ref{propHomotEquivSufConditions}}
      \newtheorem*{bproposition}{Proposition \ref{propContracibilitySufConditions}}
      \newtheorem*{cproposition}{Proposition \ref{propositionApGContractibilityLowSteps}}
      \newtheorem*{atheorem}{Theorem \ref{ultimothm}}
      \theoremstyle{definition}
	  \newtheorem{example}[theorem]{Example}
      \newtheorem{definition}[theorem]{Definition}
      
     \theoremstyle{remark}
      \newtheorem{remark}[theorem]{Remark}

 \newcommand\ZZ{{\mathbb{Z}}}
 \newcommand\NN{{\mathbb{N}}}

 \renewcommand\SS{{\mathbb{S}}}
 \renewcommand\AA{{\mathbb{A}}}

 \def\A{{\mathcal A}}

 \def\K{{\mathcal K}}
 
 \def\M{{\mathcal M}}
 
 \def\O{{\mathcal O}}

 \def\S{{\mathcal S}}

 \def\i{{\mathfrak{i}}}
 \def\s{{\mathfrak{s}}}
 
 \def\Id{{\text{Id}}}
 \def\fix{{\text{Fix}}}

\newcommand\gen[1]{\left\langle#1\right\rangle}

\usepackage{tikz}

      \makeatletter
      \def\@setcopyright{}
      \def\serieslogo@{}
      \makeatother
   
\begin{document}

\title [The homotopy types of the posets of $p$-subgroups]{The homotopy types of the posets of $p$-subgroups of a finite group}
   \author{El\'ias Gabriel Minian}
   \author{Kevin Iv\'an Piterman}
   \address{Departamento  de Matem\'atica \\IMAS-CONICET\\
 FCEyN, Universidad de Buenos Aires. Buenos Aires, Argentina.}
\email{gminian@dm.uba.ar ; kpiterman@dm.uba.ar}

\thanks{Partially supported by grants ANPCyT PICT-2011-0812, CONICET PIP 112-201101-00746 and UBACyT 20020130100369.}

   \begin{abstract}
We study the homotopy properties of the posets of $p$-subgroups $\S_p(G)$ and $\A_p(G)$ of a finite group $G$, viewed as finite topological spaces. We answer a question raised by R.E. Stong in 1984 about the relationship between the contractibility of the finite space $\A_p(G)$ and that of $\S_p(G)$ negatively, and describe the contractibility of $\A_p(G)$ in terms of algebraic properties of the group $G$. 
   \end{abstract}



\maketitle

\section{Introduction}

The poset $\S_p(G)$ of non-trivial $p$-subgroups of a finite group $G$ was introduced by K. Brown in the seventies \cite{Bro75}. Brown observed that the topology of the simplicial complex associated to this poset, which we denote by $\K(\S_p(G))$, is related to the algebraic properties of $G$, and proved the Homological Sylow Theorem
 $$\chi(\K(\S_p(G)))\equiv 1\mod |G|_p,$$ where $\chi(\K(\S_p(G)))$ denotes the Euler characteristic of the complex and  $|G|_p$ is the greatest power of $p$ that divides the order of $G$.

The study of the topological properties of $\S_p(G)$ was continued by D. Quillen in his seminal paper \cite{Qui78}. Quillen  investigated the homotopy properties of $\K(\S_p(G))$ by comparing it with the complex associated to the subposet $\A_p(G)$ of non-trivial elementary abelian $p$-subgroups of $G$. He proved that $\K(\S_p(G))$ and $\K(\A_p(G))$ are homotopy equivalent and that these polyhedra are contractible if $G$ has a non-trivial normal $p$-subgroup. Quillen conjectured that the converse should hold \cite[Conjecture 2.9]{Qui78}: if $\K(\S_p(G))$ is contractible then $\O_p(G)\neq 1$. Here $\O_p(G)$ denotes the maximal normal $p$-subgroup of $G$. The conjecture remains unproven but there have been remarkable progresses. In the nineties M. Aschbacher and S. Smith obtained the most significant partial confirmation of Quillen's conjecture so far \cite{AS93}.

The works of Brown and Quillen on the topology of the $p$-subgroup complexes have been pursued by many mathematicians, who related the topological properties of the complexes and the combinatorics of the posets with the algebraic properties of the group (see \cite{AK,Asc93,Bouc,ADR,HI,Ks1,Ks2,Seg94,Sym,TW91}). For example, in \cite{Asc93,Ks1,Ks2} the authors investigated the fundamental group of these complexes and in \cite{HI} T. Hawkes and I.M. Isaacs proved that if $G$ is $p$-solvable and has an abelian Sylow $p$-subgroup, then $\chi(\K(\S_p(G)))= 1$ if and only if $\O_p(G)\neq 1$. We refer the reader to S. Smith's book \cite{Smi11} for more details on subgroup complexes and the development of these results along the last decades.

In all the articles that we mentioned above, the authors handled the posets $\S_p(G)$ and $\A_p(G)$ topologically by means of their classifying spaces (or order complexes) $\K(\S_p(G))$ and $\K(\A_p(G))$. In 1984 R.E. Stong adopted an alternative point of view: he treated $\S_p(G)$ and $\A_p(G)$ as finite topological spaces \cite{Sto84}. Any finite poset has an intrinsic topology and in \cite{Sto84} Stong used  results on the homotopy theory of finite spaces that he obtained previously in \cite{Sto66} and results of McCord \cite{MC66}, in order to relate the (intrinsic) topology of the posets $\S_p(G)$ and $\A_p(G)$ with the algebraic properties of the group. At that time it was already known that for any finite poset $X$ (viewed as a finite space) there exists a natural weak equivalence $\mu_X:\K(X)\to X$. In particular, the posets $\A_p(G)$ and $\S_p(G)$ are weak equivalent (viewed as finite spaces) since their order complexes are homotopy equivalent. But the notion of homotopy equivalence and contractibility in the context of finite topological spaces is strictly stronger than those in the context of polyhedra. This is because the classical theorem of J.H.C. Whitehead is not longer true in the context of finite spaces (see \cite{Bar11a,BM08a}).

Stong showed that in general $\S_p(G)$ and $\A_p(G)$ are not homotopy equivalent as finite spaces. Concretely, he proved that for $G=\SS_5$, the symmetric group on five letters, and $p=2$, the finite spaces $\A_p(G)$ and $\S_p(G)$ do not have the same homotopy type. He also proved that $\S_p(G)$ is a finite contractible space if and only if $\O_p(G)\neq 1$. In this way, Quillen's Conjecture can be reformulated by saying that if $\S_p(G)$ is homotopically trivial (i.e. if $\K(\S_p(G))$ is contractible) then it is contractible (as a finite space). On the other hand, the contractibilty of the poset $\A_p(G)$ implies that of $\S_p(G)$. In a note at the end of \cite{Sto84} he left open the question whether the converse holds.

The first goal of this article is to answer Stong's question mentioned above. In this direction, we prove first the following two results.

\begin{aproposition}
	In the following cases $\S_p(G)$ and $\A_p(G)$ have the same homotopy type:
	\begin{enumerate}
		\item  $\Omega_1(G)$ is abelian for each $P\in  Syl_p(G)$, 
		\item $\A_p(G)$ is discrete, 
		\item $G=D_n$ (the dihedral group),
		\item $|G| = p^\alpha q$.
	\end{enumerate}
	Moreover, $\A_p(G)\subset \S_p(G)$ is a strong deformation retract if and only if it is a retract, and this happens if and only if condition (1) holds.
\end{aproposition}

Here $\Omega_1(G)$ denotes the subgroup generated by elements of order $p$, $|G|$ denotes the order of $G$ and $Syl_p(G)$ is the set of its Sylow $p$-subgroups.

\begin{bproposition}
	In any of the following cases, the contractibility of $\S_p(G)$ implies that of $\A_p(G)$:
	\begin{enumerate}
		\item All maximal $p$-tori are conjugated,
		\item $h(\A_p(G))\leq 1$,
		\item $|G|_p\leq p^3$.
	\end{enumerate}
\end{bproposition}

Here $h(\A_p(G))$ denotes the height of the poset $\A_p(G)$.

In order to find a counterexample to Stong's question, we used these results to discard many (thousands) potential candidates: we applied the filters provided by Proposition \ref{propHomotEquivSufConditions} and Proposition \ref{propContracibilitySufConditions} to all subgroups of order less than or equal to $576$ in the ``Small Groups library'' of GAP.   In Example \ref{contraejemplo} we exhibit a group of order $576$ which answers Stong's question negatively. We also find the minimal counterexample of a group $G$ such that $\S_p(G)$ and $\A_p(G)$ are not homotopy equivalent. In Proposition \ref{minimocardinal} we prove that if $|G|<72$, $\S_p(G)\simeq \A_p(G)$ for every prime $p$.

Our counterexample to Stong's question shows that the contractibility of the finite space $\A_p(G)$ is strictly stronger than the contractibility of $\S_p(G)$. On the other hand, by Stong's results (see Proposition \ref{stongprop} below), the contractibility of $\S_p(G)$ is equivalent to $G$ having a nontrivial normal $p$-subgroup. The second goal of this article is to understand the contractibility of $\A_p(G)$ in purely algebraic terms. In the last section of the paper we prove the following two results. The first one provides a complete answer when the poset is contractible in few steps.

\begin{cproposition}
	The followings assertions hold:
	\begin{enumerate}
		\item $\A_p(G)$ is contractible in $0$ steps if and only if $G$ has only one subgroup of order $p$, i.e. $\Omega_1(G) \simeq \ZZ_p$,
		\item $\A_p(G)$ is contractible in $1$ step if and only if $\Omega_1(G)$ is abelian,
		\item $\A_p(G)$ is contractible in $2$ steps if and only if the intersection of all maximal $p$-tori is non-trivial, if and only if $p\mid |C_G(\Omega_1(G))|$,
		\item $\A_p(G)$ is contractible in $3$ steps if and only if there exists a $p$-torus subgroup of $G$ which intersects (in a non-trivial way) every non-trivial intersection of maximal $p$-tori.
	\end{enumerate}
\end{cproposition}

The second main result of the last section provides a partial answer in the case that $\A_p(G)$ is contractible in more than $3$ steps.

\begin{atheorem}
	The poset $\A_p(G)$ is contractible in $n$ steps if and only if one of the following holds:
	\begin{enumerate}
		\item $n=0$ and $\A_p(G) = \{*\}$,
		\item $n\geq 1$ is even and $\bigcap_{A\in \M_{n-1}} A >1$,
		\item $n\geq 1$ is odd and $\gen{A:A\in \M_{n-1}}$ is abelian.
	\end{enumerate}
\end{atheorem}

All groups, posets and simplicial complexes in this paper are assumed to be finite.

\section{Preliminaries on finite spaces and R.E. Stong's approach}

We recall first some basic facts on the homotopy theory of finite topological spaces. For more details, we refer the reader to \cite{Bar11a, BM08a, MC66, Sto66}.

The standard way to study a poset $X$ topologically is by means of its order complex $\K(X)$ which is the simplicial complex of non-empty chains of $X$. However any finite poset $X$ has also an intrinsic topology where the open sets are its down-sets (recall that a down-set is a subset $U$ with the property that, if $x\in U$ and $y\leq x$ then $y\in U$). It is easy to see that a map $f:X\to Y$ between posets is order-preserving if and only if it is continuous (with the intrinsic topology), and that two continuous maps $f,g:X\to Y$ are homotopic if and only if there exists a \textit{fence} of maps $f_0,f_1,\ldots,f_n:X\to Y$ such that $f_0=f$, $f_n=g$ and $f_i$, $f_{i+1}$ are comparable for each $0\leq i<n$.

There is a relationship between the topology of $X$ and the topology of its order complex $\K(X)$, which was first discovered by McCord \cite{MC66}: there exists a natural weak equivalence $\mu_X:\K(X)\to X$. In particular they have the same homotopy groups and homology groups. This implies for example that the posets $\A_p(G)$ and $\S_p(G)$ are weak equivalent (viewed as finite topological spaces) since their order complexes are homotopy equivalent, but they are not in general homotopy equivalent as finite spaces. The classical theorem of J.H.C. Whitehead is no longer true in the context of finite spaces, and in general $\K(X)$ and $X$ are not homotopy equivalent (although they are weak equivalent). The notion of homotopy equivalence in the context of finite spaces is therefore strictly stronger than the corresponding notion in the context of simplicial complexes (see \cite{Bar11a} for examples of non-contractible finite spaces $X$ such that the corresponding $\K(X)$ are contractible). In fact, the homotopy types of finite spaces correspond to \it strong homotopy types \rm in the context of simplicial complexes (see \cite{BM12}). As we mentioned above, in \cite{Sto84} Stong proved that, for $G=\SS_5$, the symmetric group on five letters, and $p=2$, the finite spaces $\A_p(G)$ and $\S_p(G)$ do not have the same homotopy type.

The classification of homotopy types of finite spaces can be done combinatorially. This was studied by Stong in a previous article \cite{Sto66}, using  the notion of \it beat point. \rm  An element $x\in X$ is called a {\it down beat point} if $\hat{U}_x=\{ y \in X,\ y<x\}$ has a maximum, and it is an {\it up beat point} if $\hat{F}_x=\{y\in X,\ x<y\}$ has a minimum. If $x$ is a beat point (down or up), the inclusion $X-x\hookrightarrow X$ is a strong deformation retract and conversely, every strong deformation retract is obtained by removing beat points. A space without beat points is called a {\it minimal space}. Removing all beat points of $X$ leads to a minimal space called the \textit{core} of $X$. This core is unique up to homeomorphism, and two finite posets $X$ and $Y$ have the same homotopy type if and only if their cores are homeomorphic. It is easy to see that a poset with maximum or minimum is contractible.

If $G$ is a group and $X$ is a $G$-poset, instead of removing a single beat point $x$, we can remove the orbit $Gx$ and obtain an equivariant strong deformation retract $X-Gx\hookrightarrow X$. It can be shown that if $f:X\to Y$ is an equivariant map which is also a homotopy equivalence, then $f$ is in fact an equivariant homotopy equivalence \cite[Proposition 8.1.6]{Bar11a}. In particular, $X$ has a $G$-invariant core, and if, in addition, it is contractible, then $X$ has a fixed point. Using these facts, Stong proved the following.

\begin{proposition}[Stong] \label{stongprop}
	Let $G$ be a finite group and $p$ a prime number.
\begin{enumerate}
		
\item If $\A_p(G)$ is contractible then $G$ has a nontrivial normal $p$-subgroup.
\item $\S_p(G)$ is contractible if and only if $G$ has a nontrivial normal $p$-subgroup.
\end{enumerate}
\end{proposition}

 From this result one deduces that Quillen's Conjecture \cite[Conjecture 2.9]{Qui78} can be restated as follows: If $\K(\S_p(G))$ is contractible then $\S_p(G)$ is a contractible finite space (see \cite[Chapter 8]{Bar11a} for more details). 
 
 By Proposition \ref{stongprop}, if $\A_p(G)$ is contractible then $S_p(G)$ is contractible. In a note at the end of \cite{Sto84}, Stong asked whether the converse holds. We will show below that the contractibility of $\S_p(G)$ does not imply the contractibility of $\A_p(G)$ and we will exhibit the smallest counterexample.

\section{Some cases for which $\S_p(G)\simeq \A_p(G)$ and the answer to Stong's question}

In this section we prove two results that will help us to find the minimal counterexample of a group $G$ such that $\S_p(G)$ and $\A_p(G)$ do not have the same homotopy type. Recall that Stong's counterexample is $G=\SS_5$ (see \cite{Sto84}). On the other hand, we will use Propositions \ref{propHomotEquivSufConditions} and \ref{propContracibilitySufConditions} to find a counterexample that answers Stong's question  \cite[Section 3]{Sto84}.

We denote by $\Omega_1(G)$ the subgroup of $G$ generated by the elements of order $p$. The centralizer of $H$ in $G$ is denoted, as usual, by $C_G(H)$ and the center of $G$ by $Z(G)$. The set of Sylow $p$-subgroups is denoted by $Syl_p(G)$. Recall that $\O_p(G)$ denotes the intersection of all Sylow $p$-subgroups of $G$ and $\O_{p'}(G)$ denotes the largest normal subgroup of $G$ of order coprime to $p$. The \it height \rm $h(X)$ of a poset $X$ is one less than the maximum number of elements in a chain of $X$. We denote by $M(X)$ the set of maximal elements of a poset $X$.

\begin{remark}
	If $G$ is a $p$-group (or more generally, if it has a unique Sylow $p$-subgroup), then both $\S_p(G)$ and $\A_p(G)$ are contractible. Moreover $\S_p(G)$ has a maximum. On the other hand, $p\mid |C_G(\Omega_1(G))|$ because $1<Z(G)\leq C_G(\Omega_1(G))$, which implies that  $\A_p(G)$ is contractible (see Proposition \ref{propositionApGContractibilityLowSteps}).
\end{remark}

\begin{proposition}\label{propHomotEquivSufConditions}
	In the following cases $\S_p(G)$ and $\A_p(G)$ have the same homotopy type:
	\begin{enumerate}
		\item  $\Omega_1(G)$ is abelian for each $P\in  Syl_p(G)$ (see \cite[Section 3]{Sto84}),
		\item $\A_p(G)$ is discrete (i.e. $h(\A_p(G))=0$),
		\item $G=D_n$ (the dihedral group),
		\item $|G| = p^\alpha q$.
	\end{enumerate}
	Moreover, $\A_p(G)\subset \S_p(G)$ is a strong deformation retract if and only if it is a retract, and this happens if and only if condition (1) holds.
\end{proposition}

\begin{proof}
	We prove (1) and the moreover part first. If $\Omega_1(P)$ is abelian for each Sylow $p$-subgroup $P$, then $r(Q) = \Omega_1(Q)$ is a non-trivial elementary abelian subgroup. Thus $r:\S_p(G)\to\A_p(G)$ is a retraction. If $i:\A_p(G)\hookrightarrow\S_p(G)$ is the inclusion map, then $ir\leq \Id_{\S_p(G)}$ and $ri = \Id_{\A_p(G)}$. This proves (1). On the other hand, if $r:\S_p(G)\to\A_p(G)$ is a retraction, then it is easy to see that  $r(Q) \geq \Omega_1(Q)$ for each $p$-subgroup $Q$. Thus $\Omega_1(P)$ is abelian for each Sylow $p$-subgroup $P$.
	
	If the height of $\A_p(G)$ is $0$ then $\Omega_1(P)$ is abelian because there is exactly one subgroup of order $p$ in each Sylow $p$-subgroup $P$.
	
	Suppose now that $G=D_n$. If $p$ is odd there is only one Sylow $p$-subgroup. The case $p=2$ follows straightforwardly from the structure of the subgroups of dihedral groups.
	
	Finally we prove the case $|G|=p^\alpha q$. We can assume that $p\neq q$ and that the number of Sylow $p$-subgroups is $q$. If each pair of distinct Sylow $p$-subgroups has trivial intersection, then
	\[\S_p(G) = \coprod_{P\in Syl_p(G)}\S_p(P)\simeq \coprod_{P\in Syl_p(G)} \{*\} \simeq\coprod_{P\in Syl_p(G)}\A_p(P) = \A_p(G)\]
	Thus, we can suppose that there exist two distinct Sylow $p$-subgroups $P,P'$ for which $P\cap P' > 1$. By the proof of \cite[Theorem 1.36]{MI}, $\O_p(G) = P\cap P'$ for each pair of distinct Sylow $p$-subgroups $P$ and $P'$. In particular, $\S_p(G)$ is contractible and hence it remains to show that $\A_p(G)$ is contractible.
	
	Assume that $\A_p(G)$ is not contractible. Note that $\A_p(G) = \A_p(\Omega_1(G))$ and $\S_p(\Omega_1(G))\subset \S_p(G)$ is a strong deformation retract (the retraction being $R\mapsto R\cap \Omega_1(G)$). Also, $|\Omega_1(G)| = p^{\alpha'} q$, so without loss of generality  we may suppose that $G= \Omega_1(G)$.
	
	Since $\A_p(G)$ is not contractible, in particular it is not contractible in two steps and by Proposition \ref{propositionApGContractibilityLowSteps}, this implies that $p\nmid |Z(G)|$. Now we affirm that the Sylow $q$-subgroups are not normal. Otherwise, take $P$ a Sylow $p$-subgroup and $Q$ a Sylow $q$-subgroup. Then $G = PQ$ is a semidirect product. On the other hand $\O_p(G)>1$, and we can take $N$ to be a minimal non-trivial normal $p$-subgroup. Then $Q\leq C_G(N)$ and $N\leq P$ is a minimal normal $p$-subgroup. Since $P$ is a $p$-group, $N\cap Z(P) > 1$, and minimality implies that $N\leq Z(P)$, that is, $P\leq C_G(N)$. It follows that $N\leq Z(G)$, and this contradicts the fact that $p\nmid |Z(G)|$.
	In particular, we have proved that $\O_{p'}(G) = 1$ and, by the standard Hall-Higman Lemma $\O_p(G)$ is self-centralizing, i.e. $\O_p(G)\geq C_G(\O_p(G))$.

	Given $A\in \A_p(G)$, let $r(A)$ be the intersection of all  $T\in M(\A_p(G))$ such that $A\leq T$. This defines a map $r:\A_p(G)\to \A_p(G)$. We will show below that $r(A)\cap \O_p(G) >1$ for each $A\in\A_p(G)$. This implies that the map $f:\A_p(G)\to \A_p(\O_p(G))$ defined by $f(A) = r(A) \cap \O_p(G)$ verifies that
	\[if(A) \leq r(A) \geq A\]
	\[fi(A) \geq A\cap \O_p(G) = A\]
	where $i:\A_p(\O_p(G))\hookrightarrow\A_p(G)$ is the inclusion. Hence $$\A_p(G)\simeq \A_p(\O_p(G))\simeq *$$ since $\O_p(G)$ is a $p$-group.
	
	We prove then that $r(A)\cap \O_p(G) > 1$ for each $A\in\A_p(G)$. As we mentioned above, $\O_p(G)=P\cap P'$ for each pair of distinct Sylow $p$-subgroups $P,P'$. Therefore we only need to prove that $\bigcap_i B_i \cap \O_p(G) > 1$ where the $B_i$'s are all maximal $p$-tori in a same Sylow $p$-subgroup, say, $P$. Since $\O_p(G)$ is self-centralizer, $Z(P)\leq \O_p(G)$, and in particular $\Omega_1(Z(P))\leq \O_p(G)$. Consequently,
	\[1 < \Omega_1(Z(P)) = \Omega_1(Z(P))\cap \O_p(G) \leq \bigcap_{B \in M(\A_p(P))} B \cap \O_p(G)\leq \bigcap_i B_i \cap \O_p(G)\]
\end{proof}

The next proposition shows some cases where the contractibility of $\S_p(G)$ implies that of $\A_p(G)$. 

\begin{proposition}\label{propContracibilitySufConditions}
	In any of the following cases, the contractibility of $\S_p(G)$ implies that of $\A_p(G)$:
	\begin{enumerate}
		\item All maximal $p$-tori are conjugated,
		\item $h(\A_p(G))\leq 1$,
		\item $|G|_p\leq p^3$.
	\end{enumerate}
\end{proposition}

\begin{proof}
	Suppose first that all maximal $p$-tori are conjugated. We claim that the intersection of all of them is non-trivial. Indeed, if $\Omega_1(Z(\O_p(G)))\leq A$, where $A$ is a maximal $p$-torus, then $\Omega_1(Z(\O_p(G))) = \Omega_1(Z(\O_p(G)))^g\leq A^g$ for all $g\in G$ and therefore $\Omega_1(Z(\O_p(G)))$ is a non-trivial $p$-torus contained in the intersection of all maximal $p$-tori. By Proposition \ref{propositionApGContractibilityLowSteps}, $\A_p(G)$ is contractible.
	
	If $h(\A_p(G))\leq 1$, then its order complex is a graph. This implies that, in this case, $\A_p(G)$ is homotopically trivial if and only if it is contractible.
	
	If $|G|_p\leq p^3$, $h(\A_p(G)) = 0$, $1$, or $2$, and in the last case $\A_p(G) = \S_p(G)$.
\end{proof}

We exhibit now various examples. We found these examples using the GAP program \cite{Gap}. Proposition \ref{propHomotEquivSufConditions} and Proposition  \ref{propContracibilitySufConditions} were used to filter most of the groups (thousands) in the  "Small Groups library'' of GAP.

\begin{example} 
	Let $G = ((\ZZ_3\times \ZZ_3)\rtimes \ZZ_8)\rtimes \ZZ_2$ be the group with id [144,182] in the Small Groups library of GAP. Note that $|G| = 2^43^2$. If we take $p=2$, the cores of the finite spaces $\S_p(G)$ and $\A_p(G)$ have $21$ and $39$ elements respectively. In particular, they are not homotopy equivalent. Note that all maximal $p$-tori in a same Sylow $p$-subgroup of this group are conjugated, so in particular all maximal $p$-tori are conjugated in $G$. This example shows that the conditions of item (1) in Proposition \ref{propContracibilitySufConditions} are not sufficient for $\A_p(G)$ and $\S_p(G)$ to be homotopy equivalent. 
\end{example}

\begin{example}
	Let $G = \SS_3\wr \ZZ_2$, i.e. $G=(\SS_3\times \SS_3)\rtimes \ZZ_2$ where the action of $\ZZ_2$ interchanges the coordinates. The order of $G$ is $72$ and, for $p=2$, the posets $\S_p(G)$ and $\A_p(G)$ are not homotopy equivalent. This can be verified by computing their cores, which have $21$ and $39$ elements respectively.  This example shows that the conditions of item (3) in Proposition \ref{propContracibilitySufConditions} are not sufficient for $\A_p(G)$ and $\S_p(G)$ to be homotopy equivalent.

\end{example}

Surprisingly, the next group for which these posets are not homotopy equivalent is $G=\SS_5$ (and $p=2$), which is Stong's example (see \cite[Section 3]{Sto84}). We will show below that the previous example is in fact the minimum example of a group $G$ whose posets $\A_p(G)$ and $\S_p(G)$ are not homotopy equivalent (for some $p$).

In the examples of above the prime was $p=2$. In principle, this is because they have small order and can be computed from the Small Groups library of GAP. However, the next example shows that $\S_p(G)\simeq \A_p(G)$ also fails for $p >2$.

\begin{example}
	Let $G$ be the group isomorphic to
	\[(((\ZZ_2\times\ZZ_2)\times ((\ZZ_2\times \ZZ_2\times \ZZ_2\times \ZZ_2)\rtimes \ZZ_3))\rtimes \ZZ_3)\rtimes \ZZ_3\]
	which has id [1728,47861] in the Small Groups library of GAP. Its order is $1728=2^63^3$ and this is the smallest group for which $\S_p(G)$ and $\A_p(G)$ do not have the same homotopy type with a prime $p\neq 2$ ($p=3$ in this case). The cores of $\S_p(G)$ and $\A_p(G)$ have $256$ and $512$ elements respectively.
\end{example}

The following example provides the negative answer to Stong's question \cite[Section 3]{Sto84}: we exhibit a group $G$ for which $\S_p(G)$ is contractible and $\A_p(G)$ is not. 

\begin{example} \label{contraejemplo}
	Let $G$ be the subgroup of $\SS_8$ generated by the permutations $(1\,2\,8\,3)(4\,7)$ and $(1\,6\,3\,7\,8\,5)(2\,4)$. This group has order $576 = 2^63^2$ and id $[576,8654]$ in the Small Groups library of GAP. One can verify the following properties for $p=2$:
	
	\begin{enumerate}
		\item $G$ is isomorphic to $((\AA_4\times \AA_4)\rtimes \ZZ_2)\rtimes \ZZ_2$, where $\AA_4$ is the alternating group in $4$ letters.
		\item $\A_p(G)$ has height $3$.
		\item $G=\Omega_1(G)$.
		\item $\S_p(G)$ is contractible but the core of $\A_p(G)$ has $100$ elements (and therefore it is not contractible).
		\item There is a normal $p$-torus which is a maximal element in the poset $\A_p(G)$.
		\item The group $G$ is solvable by Burnside's Theorem. In particular, it satisfies Quillen's Conjecture (see \cite[Corollary 12.2]{Qui78}).
		\item Since $2^6$ does not divide $7!$, $\SS_8$ is the smallest symmetric group containing a counterexample of this type. Moreover, every subgroup of $\SS_8$ distinct of $G$ verifies that $\S_p(G)\simeq *$ implies $\A_p(G)\simeq *$.
		\item The Fitting subgroup verifies $F(G) = \O_p(G)$. In particular, $\O_p(G)$ is self-centralizing.
	\end{enumerate}
\end{example}	

	To produce this counterexample we applied the filters provided by Proposition \ref{propHomotEquivSufConditions} and Proposition \ref{propContracibilitySufConditions} to all subgroups of order less than or equal to $576$ in the Small Groups library of GAP.

We prove now that $72$ is the minimum order for which $\A_p(G)$ and $\S_p(G)$ do not have the same homotopy type (for some $p$).

\begin{proposition}\label{minimocardinal}
	If $|G|<72$ then $\S_p(G)\simeq \A_p(G)$ for each prime $p$.
\end{proposition}

\begin{proof}
	Let $1\leq n<72$ and let $G$ be a group of order $n$. If $p\nmid|G|$ both posets are empty and there is nothing to say. Otherwise, $n=p^\alpha m$ with $\alpha \geq 1$ and $(m:p)=1$. If $\alpha = 1$ or $2$, the Sylow $p$-subgroups are abelian and by Proposition \ref{propHomotEquivSufConditions}, $\A_p(G)\subset \S_p(G)$ is a strong deformation retract. If $\alpha\geq 3$ then $2^33^2= 72 > n = p^\alpha m\geq 2^3m$, and thus $1\leq m < 9$.  For $m=1$ or prime, the result follows from Proposition \ref{propHomotEquivSufConditions} (4). So it remains to show that $m\neq 4$, $6$ and $8$. If $m=4$, $6$ or $8$, as $(p:m)=1$, $p\geq 3$. But then $p^\alpha m\geq 3^34=108>72$.
\end{proof}

\section{The contractibility of $\A_p(G)$ in algebraic terms}

Example \ref{contraejemplo} shows that the contractibility of the finite space $\A_p(G)$ is strictly stronger than the contractibility of $\S_p(G)$. On the other hand, by Proposition \ref{stongprop}, the contractibility of $\S_p(G)$ is equivalent to $G$ having a nontrivial normal $p$-subgroup. Our aim is to understand the contractibility of $\A_p(G)$ in purely algebraic terms.

\begin{definition}
Let $f,g:X\to Y$ be two order-preserving maps between finite posets. We say that $f$ and $g$ are homotopic in $n$ steps (with $n\geq 0$) if there exist $f_0,\ldots,f_n:X\to Y$ such that $f=f_0$, $f_n =g$ and $f_i$, $f_{i+1}$ are comparable for every $0\leq i<n$. We denote it by $f\sim_n g$.

Two posets $X$ and $Y$ are homotopy equivalent in $n$ steps (denoted by $X\sim_n Y$) if there are maps $f:X\to Y$ and $g:Y\to X$ such that $fg\sim_n \text{Id}_Y$ and $gf\sim_n \text{Id}_X$. We say that $X$ is contractible in $n$ steps if $X\sim_n *$ (the singleton), or equivalently, there exist $x_0\in X$ and $f_0=\text{Id}_X,f_1,\ldots,f_n=c_{x_0}:X\to X$, where $c_{x_0}$ is the constant map $x_0$, such that $f_i$ and $f_{i+1}$ are comparable for each $i$.
\end{definition}

\begin{remark}\label{remarkSteps}
Note that $X\sim_0 Y$ if and only if they are homeomorphic, and that $X$ is contractible in $1$ step if and only if it has a maximum or a minimum. Note also that in this case, $X$ can be carried to a point by only removing up beat points (if it has a maximum), or down beat points (if it has a minimum). Thus, contractibility in $1$ step means that only one type of beat points is needed to be removed.
Note also that if $X\sim_n Y$ and $Y\sim_m Z$, then $X\sim_{n+m} Z$.
\end{remark}

Suppose $X$ is a contractible finite space. As we explained above, this means that there exists an ordering $x_1,\ldots,x_r$ of the elements of $X$ such that $x_i$ is a beat point of $X-\{x_1,\ldots,x_{i-1}\}$ for $i=1,\ldots,r-1$. In each step, $x_i$ can be an up beat point or a down beat point.  We say that the beat points can be \textit{removed with (at most) $n$ changes} if there are 
$1<i_1<i_2<\ldots<i_n\leq r-1$ such that all the beat points between $x_1$ and $x_{i_1-1}$, $x_{i_1}$ and $x_{i_2-1}$,\ldots,$x_{i_n}$ and $x_{r-1}$ are of the same kind. For example, if the poset $X$ has a maximum or minimum, one can reach the singleton by removing beat points without any changes (all up beat points, if it has a maximum, and all down beat points if it has a minimum).

\begin{theorem}\label{theoremStepsChangeBeatPoints}
The poset $X$ is contractible in $n$ steps if and only if we can remove the beat points with (at most) $n-1$ changes.
\end{theorem}

\begin{proof}

Assume first that there exists an ordering $\{x_1,\ldots,x_k\} = X$ such that $x_j$ is a beat point of $X_j= X-\{x_1,\ldots,x_{j-1}\}$ and that there are at most $n-1$ changes of kind of beat points.

If $n=1$, then they are all down beat points or all up beat points. Suppose the first case. 
For each $j$, let $\hat{U}^{X_j}_{x_j}=\{ x \in X_j,\ x<x_j \}$ and $y_j\in X_j$ be $y_j = \max\hat{U}^{X_j}_{x_j}$. Let $r_j:X_j\to X_{j+1}$ be the retraction which sends $x_j$ to $y_j$ and fixes the other points, and let $i_j:X_{j+1}\to X_j$ be the inclusion. Then $\alpha_1 :=i_1r_1 \leq \Id_{X_1} = \Id_X$. Let $\alpha_j = i_1i_2\ldots  i_j r_j \ldots r_2r_1:X\to X$. Since $i_jr_j\leq \Id_{X_j}$ for all $j$, we conclude that $\alpha_j\leq \Id_X$ for all $j$. In particular, for $j=k-1$, $\alpha_{k-1}\leq \Id_X$ and $\alpha_{k-1}$ is a constant map given that $r_{k-1}:X_{k-1}\to X_k = \{x_k\}$. Consequently, $X\sim_1 *$.

Now assume $n>1$ and take an ordering $\{x_1,\ldots,x_k\}=X$ of beat points with at most $n-1$ changes. Take the minimum $i$ such that  $x_i$ and $x_{i+1}$ are beat points of different kinds. By the same argument used before, it is easy to see that $X\sim_1 X-\{x_1,\ldots,x_{i}\}=X_{i-1}$ because all the beat points removed are of the same type. By induction, $X_{i-1}$ can be carried out to a point by removing beat points with at most $n-2$ changes, and then $X_{i-1}\sim_{n-1} *$. Therefore, by Remark \ref{remarkSteps}, $X\sim_n *$.

Suppose now that $X$ is contractible in $n$ steps and proceed by induction on $n$. If $n=1$, then $X$ has a maximum or a minimum. In that case we can reach the core of $X$ by removing only up beat points in the first case, or only down beat points in the latter case.

Let $n=2$ and assume, without loss of generality, that $\Id_X\leq g\geq c_{x_0}$, where $c_{x_0}$ is the constant map $x_0$. We can suppose that $X$ does not have neither a minimum nor a maximum, and this implies that $g$ is not the identity map. Let $\fix(g)$ denote the subposet of $X$ of points which are fixed by $g$. Note that $M(X)\subseteq \fix(g)\neq X$. Since $\text{Id}_X\leq g$, for any $x\in X$ we have
$$x\leq g(x)\leq g^2(x)\leq g^3(x)\leq \ldots$$
and therefore there exists $i\in \NN$ such that
$g^i(x)\in \fix(g)$.

 Take $x\in X-\fix(g)$ a maximal element. If $x<z$, then $z\in \fix(g)$ by maximality. Now, since $g\geq \Id_X$, we have $x< g(x) \leq g(z) = z$. Therefore, $x$ is an up beat point.

 Let $\{x_1,\ldots,x_k\}$ be a linear extension of $(X-\fix(g))^{op}$ and let $X_j = X-\{x_1,\ldots,x_{j-1}\}$. We affirm that $x_j$ is an up beat point of $X_j$ for each $j\geq 1$. The case $j=1$ is what we did before. Suppose $j>1$ and let $y=g^m(x_j)\in\fix(g)\subseteq X_j$. Take $z\in X_j$ such that $z>x_j$. Then $z\in \fix(g)$ and  $x_j < y=g^m(x_j)\leq g^m(z) = z$, which shows that $x_j$ is an up beat point of $X_j$. Hence, $\fix(g)$ can be obtained from $X$ by removing only up beat points. We show now that $\fix(g)$ has a minimum, using the fact that $g\geq c_{x_0}$. This implies that $\fix(g)$ can be carried out to a single point by removing only down beat points, and hence the beat points of $X$ can be removed with $1$ change (first up beat points and then down beat points). In order to see that $\fix(g)$ has a minimum, take $m\in\NN$ such that $g^m(x_0)\in\fix(g)$. Then, for any $z\in\fix(g)$ we have $z=g^{m+1}(z)\geq g^m(x_0)$.

Suppose now that $n>2$. Assume that there exists a fence $\Id_X\leq g_1\geq g_2\leq g_3\geq \ldots$, with $g_n = c_{x_0}$. Let $Y = \fix(g_1)$. We may suppose that $X\neq Y$. By the same argument used in the case $n=2$, $Y$ is obtained from $X$ by removing only up beat points. Let $i:Y\hookrightarrow X$ be the inclusion map and $r:X\to Y$ the retraction given  by the extraction of the up beat points. Then
\[\Id_Y\geq rg_2 i\leq rg_3i\geq \ldots \overset{\leq}{\geq} rg_ni = c_{rg_n(x_0)}\]
Then $Y\sim_{n-1} *$ and, by induction, the beat points of $Y$ can be removed with at most $n-2$ changes. This concludes the proof.
\end{proof}

The contractibility of $\A_p(G)$ in few steps can be described in purely algebraic terms. First we need a lemma.

\begin{lemma}\label{lemmaHomotopyCRL}
If $f,g:\A_p(G)\to \A_p(G)$ are two maps such that $\Id_{\A_p(G)}\geq f\leq g$, then $\Id_{\A_p(G)}\leq g$.
\end{lemma}

\begin{proof}
Let $A\in\A_p(G)$. We need to prove that $A\leq g(A)$. Take $a\in A$ a non-trivial element. Thus, $\gen{a}$ is a minimal $p$-torus and therefore
\[\gen{a} = f(\gen{a})\leq g(\gen{a})\leq g(A)\]
This means that $a\in g(A)$ for each non-trivial element $a\in A$. Consequently, $A\leq g(A)$.
\end{proof}

\begin{proposition}\label{propositionApGContractibilityLowSteps}
The followings assertions hold:
\begin{enumerate}
\item $\A_p(G)$ is contractible in $0$ steps if and only if $G$ has only one subgroup of order $p$, i.e. $\Omega_1(G) \simeq \ZZ_p$,
\item $\A_p(G)$ is contractible in $1$ step if and only if $\A_p(G)$ has a maximum, if and only if $\Omega_1(G)$ is abelian,
\item $\A_p(G)$ is contractible in $2$ steps if and only if the intersection of all maximal $p$-tori is non-trivial, if and only if $p\mid |C_G(\Omega_1(G))|$,
\item $\A_p(G)$ is contractible in $3$ steps if and only if there exists a $p$-torus subgroup of $G$ which intersects (in a non-trivial way) every non-trivial intersection of maximal $p$-tori.
\end{enumerate}
\end{proposition}

\begin{proof}
	Item (1) is clear and Item (2) follows from the previous lemma.
We prove (3) and (4). Assume that $\A_p(G)$ is contractible in $2$ steps. By the previous lemma we can suppose that there exists a map $f:\A_p(G)\to\A_p(G)$ with $\Id_{\A_p(G)}\leq f\geq c_N$, where $c_N$ is the constant map with value $N$, for some $N\in \A_p(G)$. In this way, if $A\in \A_p(G)$ is a maximal element, then $A\leq f(A)$ implies $A = f(A)$. Hence, for each maximal element $A$ we see that $A\geq N$, i.e. $N$ is contained in each maximal element. If $a\in N$ is a non-trivial element, then $a\in C_G(\Omega_1(G))$, which means that $p\mid |C_G(\Omega_1(G))|$. Conversely, if $p\mid |C_G(\Omega_1(G))|$, by Cauchy's Theorem there exists an element $a\in C_G(\Omega_1(G))$ of order $p$. Let $N =  \gen{a}$. Thus, $N\in \A_p(G)$ and $a\in A$ for each maximal $p$-torus $A$. If we denote by $r(B)$ the intersection of all maximal $p$-tori containing $B$, we get
\[B\leq r(B) \geq N\]
This concludes the proof of (3).

If $\A_p(G)$ is contractible in 3 steps, then by the previous lemma we can take a homotopy $\Id_{\A_p(G)}\leq f\geq g\leq c_N$, where $c_N$ is the constant map with value $N$. Moreover, $f(B)\leq r(B)$, and thus $r(B)\geq g(B)\leq N$. This means that $r(B)\cap N \geq g(B) > 1$, and therefore
\[B\leq r(B)\geq r(B)\cap N\leq N\]
is a well-defined homotopy between the identity of $\A_p(G)$ and the constant map $N$. But then $N$ intersects in a non-trivial way every non-trivial intersection of maximal elements of $\A_p(G)$. Note that this also proves the converse. 
\end{proof}

The following example shows that, unlike what happens with $\S_p(G)$ (which is always contractible in two steps since it is conically contractible \cite[Proposition 2.4]{Qui78}), the poset $\A_p(G)$ may be contractible in more than two steps.

\begin{example}
Let $G = \SS_4$. Then $|G| = 2^33$. Since $N=\gen{(1\,2)(3\,4), (1\,3)(2\,4)}$ is a non-trivial normal $2$-subgroup of $G$, both posets $\S_2(G)$ and $\A_2(G)$ are contractible by Item (3) of Proposition  \ref{propContracibilitySufConditions}. In fact, $\A_2(G)$ is contractible in $3$ steps but it is not contractible in two steps. The poset $\i(\A_p(G))$ of non-trivial intersections of maximal elements (see below for a formal definition) is given by
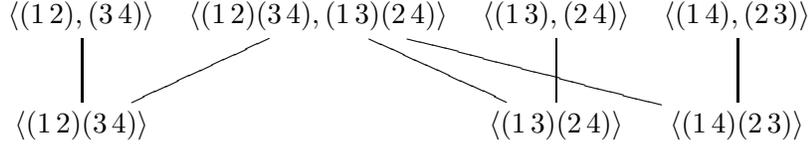
\[\xymatrix @C=.5pc{
\gen{(1\,2),(3\,4)} \ar@{-}[d] & \gen{(1\,2)(3\, 4),(1\,3)(2\, 4)} \ar@{-}[rd]  \ar@{-}[rrd]  \ar@{-}[ld]  & \gen{(1\,3),(2\, 4)} \ar@{-}[d] & \gen{(1\,4),(2\, 3)}  \ar@{-}[d] \\
\gen{(1\,2)(3\,4)} & & \gen{(1\,3)(2\, 4)} & \gen{(1\,4)(2\, 3)} 
}\]
This shows that the intersection of all maximal $p$-tori is trivial, but the subgroup $N$ intersects in a non-trivial way each non-trivial intersection of maximal $p$-tori.
\end{example}

The contractibility of $\A_p(G)$ in more than $3$ steps can be described in algebraic terms but with the aid of an extra combinatorial information of the poset. The methods that we will use are a generalization of those used in the proofs of Lemma \ref{lemmaHomotopyCRL} and Proposition \ref{propositionApGContractibilityLowSteps}.

For a lattice $L$, recall that $L^* = L-\{\hat{0},\hat{1}\}$ is called the \textit{proper part} of $L$. We say that a poset $X$ is a \textit{reduced lattice} if $X= L^*$ for some lattice $L$. Equivalently, for every pair of elements $\{x,y\}$ with an upper bound in $X$ there exists the supremum $x\vee y$. This condition is equivalent to saying that for each pair of elements $\{x,y\}$ with a lower bound in $X$ there exists the infimum $x\wedge y$. Recall that $M(X)$ denotes the set of maximal elements of $X$, similarly we denote by $m(X)$ the minimal elements. If $x\in X$, we denote by $M(x)$ the set of maximal elements over $x$ and by $m(x)$ the set of minimal elements below $x$. A reduced lattice $X$ is \textit{atomic} if every element is the supremum of the minimal elements below it, i.e. if $x = \bigvee_{y\in m(x)} y$ for each $x\in X$. Similarly, $X$ is \textit{coatomic} if $X^{op}$ is atomic, i.e. if $x = \bigwedge_{y\in M(x)} y$ for each $x\in X$.

The poset $\A_p(G)$ is an atomic reduced lattice: the infimum of two $p$-tori is their intersection and the supremum is the subgroup generated by both subgroups.

Given two order preserving maps $f,g:X\to Y$, where $Y$ ia a reduced lattice, such that $\{f(a),g(a)\}$ is lower bounded (resp. upper bounded) for each $a\in X$, we define the maps $f\wedge g, f\vee g:X\to Y$ by $(f\wedge g)(a) = f(a)\wedge g(a)$ and $(f\vee g)(a) = f(a)\vee g(a)$.

\begin{proposition}\label{propFenceAtomic}
Let $X$ be an atomic reduced lattice. If $\Id_X\sim_n g$, then there exist $f_0,\ldots,f_n:X\to X$ with
\[\Id_X = f_0\leq f_1\geq f_2\leq \ldots\overset{\geq}{\leq} f_n = g\]
and such that $f_{2k} = f_{2k-1}\wedge f_{2k+1}$ for each $1\leq k<n/2$ and $f_{2k+1} = f_{2k}\vee f_{2k+2}$ for each $0\leq k<n/2$.
\end{proposition}

\begin{proof}
	Note first that, since $X$ is an atomic reduced lattice, one can use the same argument as in the proof of  Lemma \ref{lemmaHomotopyCRL} to show that, if $\Id_X\geq \tilde{f}_1\leq \tilde{f}_2$, then $\Id_X\leq \tilde{f}_2$. Then there exists a fence
	
	\[\Id_X = f_0\leq f_1\geq f_2\leq \ldots\overset{\geq}{\leq} f_n = g\]

Now for $i$ even we have $f_{i-1}\geq f_i\leq f_{i+1}$, and we can replace  $f_i$ by $f_{i-1}\wedge f_{i+1}$ and obtain $f_{i-1}\geq f_{i-1}\wedge f_{i+1}\leq f_{i+1}$. Then we can proceed analogously with all odd indexes $i$.
\end{proof}

The following constructions were introduced by J. Barmak in \cite[Chapter 9]{Bar11a}. Given a reduced lattice $X$, let \[\i(X) = \left\{\bigwedge_{x\in S} x: S\subseteq M(X),\, S\neq\emptyset \text{ and lower bounded}\right\}\]  \[\s(X) = \left\{\bigvee_{x\in S} x: S\subseteq m(X),\, S\neq\emptyset \text{ and upper bounded}\right\}\]

With these notations, $X$ is atomic if and only if $X = \s(X)$, and it is coatomic if and only if $X=\i(X)$. Both $\i(X)$ and $\s(X)$ are strong deformation retracts of $X$ (see \cite[Chapter 9]{Bar11a}). Moreover, $\i(X)$ can be obtained from $X$ by extracting only up beat points, and $\s(X)$ by extracting only down beat points. As $\i\i(X) = \i(X)$ and $\s\s(X) = \s(X)$, we can perform these two operations until we obtain a core of $X$. In particular, the core of $X$ is both an atomic and coatomic reduced lattice. Let $n\geq 0$. If $X$ is atomic and $n\geq 0$, denote by $X_n$ the $(n+1)$-th term in the sequence \[X\supseteq \i(X)\supseteq\s\i(X)\supseteq\i\s\i(X)\supseteq\ldots\]

In the same way, when $X$ is coatomic denote by $X_n$ the $(n+1)$-th term in the sequence
\[X\supseteq \s(X)\supseteq\i\s(X)\supseteq\s\i\s(X)\supseteq\ldots\]
Note that if $X$ is a $G$-poset, then $\i(X)$ and $\s(X)$ are $G$-invariant. So this method provides an easy tool to find a $G$-invariant core of $X$.

\begin{remark}
	Note that if $X$ is a reduced lattice and $\Id_X\leq f$, then $f(x)\leq \bigwedge_{y\in M(x)} y$ for any $x\in X$.
\end{remark}

\begin{theorem}\label{theoremStepContractibilty}
Let $X$ be an atomic reduced lattice. The following conditions are equivalent:
\begin{enumerate}
\item $X\sim_n *$,
\item $\i(X)\sim_{n-1} *$,
\item $X_i\sim_{n-i} *$ for all $i\geq 0$,
\item $X_n = *$.
\end{enumerate}

With the convention that, for a negative number $m$, $X\sim_m*$ means that $X\sim_0 *$. Analogous equivalences hold when $X$ is a coatomic reduced lattice, with $\s(X)$ instead of $\i(X)$.
\end{theorem}

\begin{proof}
 (1) $\Rightarrow$ (2). Let $r:X\to\i(X)$ be the retraction $r(x) = \bigwedge_{y\in M(x)} y$ and let $i:\i(X)\hookrightarrow X$ be the inclusion map. Then $ri = \Id_{\i(X)}$ and $ir\geq \Id_X$. If $X\sim_n * $, by Proposition \ref{propFenceAtomic} there exist $f_1,\ldots,f_n:X\to X$ such that
\[\Id_X\leq f_1\geq f_2\leq f_3\geq f_4\ldots\]
with $f_n$ a constant map. By the previous remark, $f_1\leq ir$, and therefore we have a fence
\[\Id_X\leq ir \geq f_2\leq f_3\geq f_4\ldots\]
Composing with $i$ and  $r$ we have:
\[\Id_{\i(X)} = ri \leq \Id_{\i(X)}=riri \geq rf_2i\leq rf_3i\geq rf_4i\leq \ldots\]
Hence $\i(X)\sim_{n-1}*$

(2) $\Rightarrow$  (1). Since $\i(X)$ is obtained from $X$ by removing only up beat points, and $\i(X)\sim_{n-1} *$, by Theorem \ref{theoremStepsChangeBeatPoints} $X\sim_n *$.

(1)  $\Leftrightarrow$ (3). This follows by applying induction, the arguments of above and the analogous results for coatomic reduced lattices.

(4)  $\Rightarrow$ (1). This follows from Theorem \ref{theoremStepsChangeBeatPoints} and the fact that each time that we apply $\i$ or $\s$, we perform a change of kind of beat points. 

(3) $\Rightarrow$ (4). Straightforward.
\end{proof}

\begin{remark}
If $X$ is atomic, then $X_n$ is coatomic for $n$ odd and it is atomic for  $n$ even. In particular, if $X\sim_n*$, by the previous theorem $X_n = *$, which means that $X_{n-1}$ has a maximum if $n$ is odd, or it has a minimum if $n$ is even. Thus if we let $\M_n$ to be $m(X_n)$ for $n$ even and $M(X_n)$ for $n$ odd, we conclude that $X\sim_n*$ if and only if $|\M_n| = 1$.
\end{remark}

Now we can apply these results to describe the contractibility in steps of $\A_p(G)$ in algebraic terms.

\begin{theorem}\label{ultimothm}
The poset $\A_p(G)$ is contractible in $n$ steps if and only if one of the following holds:
\begin{enumerate}
\item $n=0$ and $\A_p(G) = \{*\}$,
\item $n\geq 1$ is even and $\bigcap_{A\in \M_{n-1}} A >1$,
\item $n\geq 1$ is odd and $\gen{A:A\in \M_{n-1}}$ is abelian.
\end{enumerate}
\end{theorem}

\begin{proof}
By the previous remark $\A_p(G)\sim_n*$ if and only if $|\M_n|=1$. 

If $n$ is odd, $\A_p(G)_{n-1}$ has a maximum and $\M_{n-1}$ is the set of minimal elements of $\A_p(G)_{n-1}$. If $B\in\A_p(G)_{n-1}$ is the maximum,  $B\geq A$ for each $A\in \M_{n-1}$ and then $\gen{A:A\in \M_{n-1}}\leq B$ is an abelian subgroup.

If $n$ is even, $\A_p(G)_{n-1}$ has a minimum and $\M_{n-1}$ is the set of maximal elements of $\A_p(G)_{n-1}$. If $B\in\A_p(G)_{n-1}$ is the minimum,  $B\leq A$ for each $A\in \M_{n-1}$ and then $1<B\leq \bigcap_{A\in \M_{n-1}}A$ is a non-trivial subgroup. This proves the ``if'' part. 

For the ``only if'' part, note that in either case we have that $\A_p(G)_{n-1}$ has a maximum or a minimum, thus it is contractible in $1$ step and the result follows from Theorem \ref{theoremStepContractibilty}.
\end{proof}



\end{document}